\documentclass[11pt]{amsart}

\usepackage{fullpage} 
\usepackage[foot]{amsaddr} 
\usepackage{graphicx} 
\usepackage{amsmath, amssymb, amsfonts, amsthm, mathtools} 
\usepackage{enumerate} 
\usepackage{caption} 

\usepackage{color} 
\usepackage{url} 
\usepackage[plainpages = false, colorlinks, citecolor = blue, filecolor = black, linkcolor = red, urlcolor = cyan]{hyperref} 

\theoremstyle{definition} 
\newtheorem{definition}{Definition}[section] 
\theoremstyle{plain} 
\newtheorem{theorem}[definition]{Theorem}

\newtheorem{lemma}[definition]{Lemma}
\newtheorem{corollary}[definition]{Corollary}

\allowdisplaybreaks 

\usepackage[backend = bibtex, giveninits = true, style = alphabetic, natbib = true, url = false, sorting = nyt, doi = true, abbreviate = true, backref = false]{biblatex}
\renewbibmacro{in:}{\ifentrytype{article}{}{\printtext{\bibstring{in}\intitlepunct}}}

\newcommand{\R}{\mathbb{R}} 
\newcommand{\N}{\mathbb{N}} 
\newcommand{\ie}{{\it{i.e.}, }} 
 
\newcommand{\vareps}{\varepsilon} 
\newcommand{\del}{\partial} 
\newcommand{\n}{\mathbf{n}}  
\newcommand{\tn}{\textnormal} 
\newcommand{\bs}{\boldsymbol} 
\renewcommand{\vec}[1]{\hat{\bs{#1}}} 
\newcommand{\abs}[1]{\left\lvert{#1}\right\rvert}
\renewcommand{\H}{\mathbf{H}}

\begin{document} 
\title{Analyticity of Steklov Eigenvalues of Nearly-Hyperspherical Domains in $\R^{d + 1}$} 
\author{Chee Han Tan}
\address{Department of Mathematics, Wake Forest University, Winston-Salem, NC} 
\email{tanch@wfu.edu} 

\author{Robert Viator} 
\address{Department of Mathematics and Statistics, Swarthmore College, Swarthmore, PA} 
\email{rviator1@swarthmore.edu} 


\date{\today} 

\begin{abstract} 
We consider the Dirichlet-to-Neumann operator (DNO) on nearly-hyperspherical domains in dimension greater than 3.  Treating such domains as perturbations of the ball, we prove the analytic dependence of the DNO on the shape perturbation parameter for fixed perturbation functions.  Consequently, we conclude that the Steklov eigenvalues are analytic in the shape perturbation parameter as well.  To obtain these results, we use the strategy of Nicholls and Nigam (2004), and of Viator and Osting (2020); we transform the Laplace-Dirichlet problem on the perturbed domain to a more complicated, parameter-dependent equation on the ball, and then geometrically bound the Neumann expansion of the transformed DNO.  These results are a generalization of the work of Viator and Osting (2020) for dimension 2 and 3.
\end{abstract} 

\maketitle 

\section{Introduction} 
For any integer $d\ge 3$, we call $\Omega_\vareps \subset \R^{d+1}$ a \emph{nearly-hyperspherical domain} if it is a small perturbation of the unit $(d + 1)$-ball $B$ in $\R^{d +1}$, given by
\begin{align*} 
\Omega_\varepsilon  & = \left\{(r, \hat{\theta}): 0 \leq r \leq 1+ \varepsilon \rho(\hat{\theta}), \, \hat{\theta} \in S^d \right\}. 
\end{align*} 
Here, $\rho\in C^{s + 2}(S^d)$ is the \emph{domain perturbation function} for some $s\in\N = \{0, 1, 2, \dots\}$, $S^d$ is the unit $d$-sphere in $\R^{d + 1}$, and $\vareps\ge 0$ is the \emph{perturbation parameter} which we assumed to be small in magnitude.

For fixed $\rho \in C^{s+2}(S^d)$, we consider the Steklov eigenvalue problem on the nearly-hyperspherical domain $\Omega_\varepsilon$: 
\begin{equation} \label{Steklov} 
\begin{alignedat}{2}
\Delta u_\vareps & = 0 && \qquad  \tn{ in } \Omega_\vareps, \\
\del_{\n_{\rho, \vareps}} u_\vareps & = \sigma_\vareps u_\vareps && \qquad \tn{ on } \del\Omega_\vareps, 
\end{alignedat}
\end{equation} 
where $\del_{\n_{\rho, \vareps}} = \n_{\rho, \vareps} \cdot \nabla$ is the unit outward normal derivative on the boundary $\partial \Omega_\varepsilon$ of the perturbed domain $\Omega_\varepsilon$. It is well-known \cite{girouard2017} that the Steklov spectrum is discrete, real, and nonnegative, and we arrange them in non-decreasing order $$0 = \sigma_0(\Omega_\vareps)< \sigma_1(\Omega_\varepsilon) \leq \sigma_2(\Omega_\vareps) \leq \dots ,$$ tending to infinity. The Steklov spectrum matches the spectrum of the Dirichlet-to-Neumann operator (DNO), $G_{\rho, \vareps}\colon H^{s + \frac{1}{2}}(\del\Omega_\vareps)\to H^{s - \frac{1}{2}}(\del\Omega_\vareps)$, given by
\begin{equation*}
\xi\mapsto G_{\rho, \vareps}\xi = \n_{\rho, \vareps}\cdot \nabla v_\vareps\rvert_{r = 1 + \varepsilon\rho}, 
\end{equation*} 
where $v_\vareps$ is the \emph{harmonic extension} of $\xi$ to $\Omega_\vareps$, satisfying 
\begin{equation} \label{pertHarmExt}
\begin{alignedat}{2} 
\Delta v_\vareps & = 0 && \qquad \tn{ in } \Omega_\vareps, \\ 
v_\vareps & = \xi && \qquad \tn{ on } \del\Omega_\vareps. 
\end{alignedat}
\end{equation} 

In previous work \cite{viator2020}, it was proven that $\sigma_\vareps$ is analytic with respect to the perturbation parameter $\vareps > 0$ for nearly-circular domain ($d = 1$) and nearly-spherical domain ($d = 2$). The goal of this paper is to show that the same result holds for $d\ge 3$. 

\begin{theorem} \label{thm:DNOAnalytic}
Let $d\ge 3$ and $s\in\N$. If $\rho\in C^{s + 2}(S^d)$, then the Dirichlet-to-Neumann operator, $G_{\rho, \vareps}\colon H^{s + \frac{3}{2}}(\del\Omega_\vareps)\to H^{s + \frac{1}{2}}(\del\Omega_\vareps)$, is analytic in the perturbation parameter $\vareps$. More precisely, if $\rho\in C^{s + 2}(S^d)$, then there exists a Neumann series 
\begin{equation*} 
G_{\rho, \vareps}(\xi) = \sum_{n = 0}^\infty \vareps^nG_{\rho, n}(\xi) 
\end{equation*} 
that converges strongly as an operator from $H^{s + \frac{3}{2}}(S^d)$ to $H^{s + \frac{1}{2}}(S^d)$. That is, there exists constants $K_1 = K_1(B, d, s) > 0$ and $\alpha = \alpha(d) > 1$ such that 
\begin{equation*}
\|G_{\rho, n}\xi\|_{H^{s + \frac{1}{2}}(S^d)}\le K_1\|\xi\|_{H^{s + \frac{3}{2}}(S^d)}A^n 
\end{equation*}
for $A > \alpha\max\{2K_0C_0\abs{\rho}_{C^{s + 2}}, M\abs{\rho}_{C^{s + 2}}\}$.
\end{theorem} 

We prove Theorem \ref{thm:DNOAnalytic} for all $d \geq 3$ simultaneously: the proof can be found in section \ref{sec:DNO}. The proof follows the strategy of \cite{nicholls2004, viator2020}. We first show the analyticity of the harmonic extension, \ie that for fixed $\xi : \partial \Omega_\varepsilon \rightarrow \mathbb{C}$ the solution $v_\varepsilon$ to \eqref{pertHarmExt} is analytic in $\vareps > 0$; see Section \ref{sec:HarmExtAnaly}. Using this and a recursive formula for the DNO $G_{\rho, \varepsilon}$, we then prove that $G_{\rho, \varepsilon}$ also depends analytically on $\vareps > 0$, establishing Theorem \ref{thm:DNOAnalytic}.

The next corollary follows from \cite{Kato} and Theorem \ref{thm:DNOAnalytic}, establishing the analytic dependence on $\varepsilon$ of the Steklov eigenvalues $\sigma_\varepsilon$ of \eqref{Steklov} within the same disk of convergence as in Theorem \ref{thm:DNOAnalytic}:

\begin{corollary} \label{thm:Steklov} 
The Steklov eigenvalues $\sigma_\varepsilon$ of \eqref{Steklov} consist of branches of one or several analytic functions which have at most algebraic singularities near $\varepsilon = 0$.  The same is true of the corresponding eigenprojections.
\end{corollary} 

The motivation of this paper is to lay the groundwork for the asymptotic study of Steklov eigenvalues of nearly-hyperspherical domains in $\mathbb{R}^{d+1}$ for $d\ge 3$.  Asymptotic and perturbative methods have already yielded meaningful results in the study of Steklov shape optimization in two and three dimensions \cite{viator2018, viator2022}. There, the authors begin with a nearly-circular domain in $\mathbb{R}^2$ or a nearly-spherical domain in $\mathbb{R}^3$ and, using the analyticity result proved in \cite{viator2020}, demonstrate that the ball is \emph{not} a shape optimizer for a large class of Steklov eigenvalues, while also obtaining a collection of Steklov eigenvalues which \emph{are} (locally) optimized by the ball.  This paper provides the foundation for future work using similar techniques for nearly-hyperspherical domains in dimension greater than three.


\section{Laplacian in Hyperspherical Coordinates in $\R^{d + 1}$} 
Given $d\ge 1$, let $\left(x_1, x_2, \dots, x_{d + 1}\right)$ denote the $(d + 1)$-dimensional Cartesian coordinates. We define the $(d + 1)$-dimensional hyperspherical coordinates $\left(r, \theta_1, \theta_2, \dots, \theta_d\right)$ as follows: 
\begin{align*} 
x_1 & = r\cos\theta_1\sin\theta_2\sin\theta_3\dots \sin\theta_d, \\ 
x_2 & = r\sin\theta_1\sin\theta_2\sin\theta_3\dots \sin\theta_d, \\ 
x_j & = r\cos\theta_{j - 1}\sin\theta_j\sin\theta_{j + 1}\dots \sin\theta_d, \ \ j = 3, 4, \dots, d, \\ 
x_{d + 1} & = r\cos\theta_d, 
\end{align*} 
where $r\ge 0$ is the radius of a $(d + 1)$-dimensional sphere, $0\le \theta_1\le 2\pi$ is the azimuth, and $0\le \theta_2, \theta_3, \dots, \theta_d\le \pi$ are the inclinations. In particular, these correspond to polar coordinates $(x_1, x_2) = (r\cos\theta_1, r\sin\theta_1)$ for $d = 1$ and spherical coordinates $(x_1, x_2, x_3) = (r\cos\theta_1\sin\theta_2, r\sin\theta_1\sin\theta_2, r\cos\theta_2)$ for $d = 2$. 

The hyperspherical coordinates are an orthogonal curvilinear coordinate system in $\R^{d + 1}$. The associated metric tensor $g$ is therefore diagonal, with components 
\begin{equation*} 
g_{ij} = \sum_{k = 1}^{d + 1} \frac{\del x_k}{\del\theta_i}\frac{\del x_k}{\del\theta_j} = h_i^2\delta_{ij}, \ \ 0\le i, j\le d,  
\end{equation*} 
where $\theta_0 = r$ and the scale factors are defined as $h_0 = 1$ and $h_i = r\prod_{k = i + 1}^d \sin\theta_k$ for $i = 1, 2, \dots, d$; the latter includes the empty product, which gives $h_d = r$. The volume element $dV$ in hyperspherical coordinates is given by 
\begin{equation} \label{eq:SCvolume}
dV = \sqrt{\abs{\det{g}}}\, dr\, d\theta_1\dots d\theta_d = \left(r^d\prod_{k = 2}^d \sin^{k - 1}\theta_k\right) dr\, d\theta_1\dots d\theta_d. 
\end{equation} 

For notational convenience, we introduce new variables $\eta_i = h_i/r$ for $i = 1, 2, \dots, d$ and $h = \sqrt{\abs{\det g}}$. The gradient operator in hyperspherical coordinates is defined as 
\begin{equation} \label{eq:SCgrad} 
\nabla = \sum_{i = 0}^d \frac{1}{h_i}\frac{\del}{\del\theta_i}\vec{\theta}_i = \frac{\del}{\del r}\vec{r} + \frac{1}{r}\sum_{i = 1}^d \frac{1}{\eta_i}\frac{\del}{\del\theta_i}\vec{\theta}_i = \frac{\del}{\del r}\vec{r} + \frac{1}{r}\nabla_{S^d}, 
\end{equation} 
where $\vec{r}, \vec{\theta}_1, \vec{\theta}_2, \dots, \vec{\theta}_d$ are orthonormal hyperspherical basis vectors. The Laplacian is given in terms of the metric tensor by 
\begin{equation} \label{eq:SClaplacian_metric} 
\Delta = \frac{1}{h}\sum_{i = 0}^d \frac{\del}{\del\theta_i}\left(\frac{h}{h_i^2}\frac{\del}{\del\theta_i}\right). 
\end{equation} 
Since $h_i$ is independent of $\theta_i$ and $h = h_0h_1\dots h_d$ is separable in hyperspherical coordinates, \eqref{eq:SClaplacian_metric} simplifies to 
\begin{align*} 
\Delta & = \sum_{i = 0}^d \frac{1}{h_i^2}\cdot \frac{1}{h}\frac{\del}{\del\theta_i}\left(h\frac{\del}{\del\theta_i}\right) = \frac{1}{r^d}\frac{\del}{\del r}\left(r^d\frac{\del}{\del r}\right) + \frac{1}{r^2}\Delta_{S^d},  
\end{align*} 
where $\Delta_{S^d}$ is the spherical Laplacian on the unit $d$-sphere $S^d$, given by 
\begin{equation} \label{eq:SCBeltrami} 
\Delta_{S^d} = \sum_{i = 1}^d \frac{1}{\eta_i^2\sin^{i - 1}\theta_i}\frac{\del}{\del\theta_i}\left(\sin^{i - 1}\theta_i\frac{\del}{\del\theta_i}\right). 
\end{equation} 
For $d = 2$, we recover the Laplacian in spherical coordinates $(r, \theta_1, \theta_2) = (r, \phi, \theta)$: 
\begin{equation*}
\Delta = \frac{1}{r^2}\frac{\del}{\del r}\left(r^2\frac{\del}{\del r}\right) + \frac{1}{r^2}\left[\frac{1}{\sin^2\theta}\frac{\del^2}{\del\phi^2} + \frac{1}{\sin\theta}\frac{\del}{\del\theta}\left(\sin\theta\frac{\del}{\del\theta}\right)\right]. 
\end{equation*}


\section{Change of Variables} 
For notational convenience, let $\del_r$ and $\del_i$ denote the partial derivative with respect to $r$ and $\theta_i$, respectively, for $i = 1, 2, \dots, d$. We first consider the problem of harmonically extending a function $\xi(\hat\theta)$ from $\del\Omega_\vareps$ to $\Omega_\vareps$: 
\begin{equation} \label{eq:HarmExt} 
\begin{aligned}
\frac{1}{r^d}\del_r\left(r^d\del_rv\right) + \frac{1}{r^2}\Delta_{S^d}v & = 0 && \qquad \tn{ in } \Omega_\vareps, \\ 
v & = \xi && \qquad \tn{ on } \del\Omega_\vareps. 
\end{aligned} 
\end{equation} 
Following \cite{nicholls2001, nicholls2004, viator2020}, we introduce the change of variables 
\begin{equation} \label{eq:CoV} 
(r', \theta') = \left(\frac{r}{1 + \vareps\rho(\hat\theta)}, \hat\theta\right) 
\end{equation} 
which maps $\Omega_\vareps$ to the unit $(d + 1)$-ball $B\coloneqq \Omega_0$. The partial derivatives in the new coordinates are given by 
\[ \del_r = \frac{1}{1 + \vareps\rho(\theta')}\del_{r'}, \qquad \del_i = \del_{i'} - \frac{\vareps r'\del_{i'}\rho(\theta')}{1 + \vareps\rho(\theta')}\del_{r'}, \ \ i = 1, 2, \dots, d. \] 
The harmonic extension $v$ transforms to  
\begin{align*} 
u_\vareps(r', \theta') = v\left((1 + \vareps\rho(\theta'))r', \theta'\right). 
\end{align*} 
The radial component of the Laplacian transforms to 
\begin{align*} 
\frac{1}{r^d}\del_r\left(r^d\del_rv\right) & = \frac{1}{(1 + \vareps\rho)^d(r')^d}\cdot \frac{1}{1 + \vareps\rho}\del_{r'}\left((1 + \vareps\rho)^d(r')^d\cdot \frac{1}{1 + \vareps\rho}\del_{r'}u_\vareps\right) \\ 
& = \frac{1}{(1 + \vareps\rho)^2(r')^d}\del_{r'}\left((r')^d\del_{r'}u_\vareps\right). 
\end{align*} 
Referring to \eqref{eq:SCBeltrami}, the angular component of the Laplacian transforms to 
\begin{align*} 
\del_i \left(\sin^{i - 1}\theta_i\, \del_iv\right) & = \del_{i'}\left(\sin^{i - 1}\theta_i'\, \del_{i'}u_\vareps\right) - \frac{\vareps r'\sin^{i - 1}\theta_i'\, \del_{i'}\rho}{1 + \vareps\rho}\del_{r'}\del_{i'}u_\vareps \\ 
& \qquad - \vareps r'\del_{i'}\left(\frac{\sin^{i - 1}\theta_i'\, \del_{i'}\rho}{1 + \vareps\rho}\del_{r'}u_\vareps\right) + \frac{\vareps^2\sin^{i - 1}\theta_i'\, (\del_{i'}\rho)^2}{(1 + \vareps\rho)^2} r'\del_{r'}\left(r'\del_{r'}u_\vareps\right) \\ 
& = \del_{i'}\left(\sin^{i - 1}\theta_i'\, \del_{i'}u_\vareps\right) - \frac{\vareps r'\sin^{i - 1}\theta_i'\, \del_{i'}\rho}{1 + \vareps\rho}\del_{r'}\del_{i'}u_\vareps \\ 
& \qquad -\vareps r'\left[{\frac{\sin^{i - 1}\theta_i'\, \del_{i'}\rho}{1 + \vareps\rho}\del_{i'}\del_{r'}u_\vareps} + \left[\frac{\del_{i'}\left(\sin^{i - 1}\theta_i'\, \del_{i'}\rho\right)}{1 + \vareps\rho} - \frac{\vareps \sin^{i - 1}\theta_i'\, (\del_{i'}\rho)^2}{(1 + \vareps\rho)^2}\right]\del_{r'}u_\vareps\right] \\ 
& \qquad + \frac{\vareps^2\sin^{i - 1}\theta_i'\, (\del_{i'}\rho)^2}{(1 + \vareps\rho)^2}\Big[(r')^2\del_{r'}^2u_\vareps + r'\del_{r'}u_\vareps\Big] \\ 
& = \del_{i'}\left(\sin^{i - 1}\theta_i'\, \del_{i'}u_\vareps\right) - \frac{\vareps r'\del_{i'}\left(\sin^{i - 1}\theta_i'\, \del_{i'}\rho\right)}{1 + \vareps\rho}\del_{r'}u_\vareps \\ 
& \qquad - \frac{2\vareps r'\sin^{i - 1}\theta_i'\, \del_{i'}\rho}{1 + \vareps\rho}\del_{i'}\del_{r'}u_\vareps + \frac{\vareps^2(r')^2\sin^{i - 1}\theta_i'\, (\del_{i'}\rho)^2}{(1 + \vareps\rho)^2}\left[\del_{r'}^2u_\vareps + \frac{2}{r'}\del_{r'}u_\vareps\right]. 
\end{align*} 
Substituting the transformed radial and angular components into the Laplace's equation \eqref{eq:HarmExt}, substituting $r^{-2} = (r')^{-2}(1 + \vareps\rho)^{-2}$ in front of $\Delta_{S^d}$, multiplying by $(1 + \vareps\rho)^4$, and dropping the primes on the transformed variables, we obtain 
\begin{align*} 
0 & = (1 + \vareps\rho)^4\Delta v \\ 
& = (1 + \vareps\rho)^2\Delta u_\vareps - \frac{\vareps(1 + \vareps\rho)}{r}(\Delta_{S^d}\rho)\del_ru_\vareps - \frac{2\vareps(1 + \vareps\rho)}{r}\sum_{i = 1}^d \frac{\del_i\rho}{\eta_i^2}\del_i\del_ru_\vareps \\ 
& \qquad + \vareps^2\sum_{i = 1}^d \frac{(\del_i\rho)^2}{\eta_i^2}\Big[\del_r^2u_\vareps + \frac{2}{r}\del_ru_\vareps\Big]. 
\end{align*} 
Finally, collecting terms of order $\vareps$ and $\vareps^2$, we see that $u_\vareps$ satisfies the following transformed harmonic extension problem: 
\begin{equation} \label{eq:HarmExtTran} 
\begin{aligned} 
\Delta u_\vareps & = \vareps L_1u_\vareps + \vareps^2L_2u_\vareps && \qquad \tn{ in } B, \\ 
u_\vareps(1, \theta) & = \xi(\theta) && \qquad \tn{ on } S^d, 
\end{aligned} 
\end{equation} 
where $\theta = (\theta_1, \theta_2, \dots, \theta_d)$ and 
\begin{align*} 
L_1u_\vareps & = -2\rho\Delta u_\vareps + (\Delta_{S^d}\rho)r^{-1}\del_ru_\vareps + 2\sum_{i = 1}^d \frac{\del_i\rho}{\eta_i^2} r^{-1}\del_i\del_ru_\vareps, \\ 
L_2u_\vareps & = -\rho^2\Delta u_\vareps + (\rho\Delta_{S^d}\rho)r^{-1}\del_ru_\vareps + 2\rho\sum_{i = 1}^d \frac{\del_i\rho}{\eta_i^2} r^{-1}\del_i\del_ru_\vareps - \sum_{i = 1}^d \frac{(\del_i\rho)^2}{\eta_i^2}\Big[\del_r^2u_\vareps + 2r^{-1}\del_ru_\vareps\Big] \\ 
& = \rho^2\Delta u_\vareps + \rho L_1u_\vareps - \sum_{i = 1}^d \frac{(\del_i\rho)^2}{\eta_i^2}\Big[\del_r^2u_\vareps + 2r^{-1}\del_ru_\vareps\Big]. 
\end{align*}


\section{Analyticity of Harmonic Extension} \label{sec:HarmExtAnaly} 
In this section we show that the solution $u_\vareps$ of the transformed harmonic extension problem \eqref{eq:HarmExtTran} is analytic with respect to $\vareps$. We begin by formally expanding $u_\vareps$ as a power series in $\vareps$: 
\begin{equation} \label{eq:HarmExtPS}
u_\vareps(r, \theta) = \sum_{n = 0}^\infty u_n(r, \theta)\vareps^n. 
\end{equation} 
Substituting \eqref{eq:HarmExtPS} into \eqref{eq:HarmExtTran} and collecting terms in powers of $\vareps$, we see that $u_0$ satisfies 
\begin{equation} \label{eq:HarmExtPS0}
\begin{aligned}  
\Delta u_0 & = 0 && \qquad \tn{ in } B, \\ 
u_0(1, \theta) & = \xi(\theta) && \qquad \tn{ on } S^d, 
\end{aligned} 
\end{equation} 
and $u_n$ for $n = 1, 2, \dots$ satisfies the following recursive formula: 
\begin{equation} \label{eq:HarmExtPS1}
\begin{aligned}  
\Delta u_n & = L_1u_{n - 1} + L_2u_{n - 2} && \qquad \tn{ in } B, \\ 
u_n(1, \theta) & = 0 && \qquad \tn{ on } S^d. 
\end{aligned} 
\end{equation}

Our first lemma establishes an elliptic estimate for the Poisson's equation with Dirichlet boundary data, which is analogous to \cite[Lemma~3.1]{viator2020}. We now review some basic facts about hyperspherical harmonics before stating Lemma \ref{lemma:Elliptic}; see \cite{Avery} for more details.  

For $\ell = 0, 1, 2, \dots$, a hyperspherical harmonic of order $\ell$ on $S^d$ is the restriction to $S^d$ of a homogeneous harmonic polynomial of degree $\ell$. Let $\H_\ell^d$ denote the space of all hyperspherical harmonics of order $\ell$ on $S^d$. The space $\H_\ell^d$ has an orthonormal basis $\{Y_\ell^m(\theta)\}_{m = 1}^{N(d, \ell)}$ with respect to the $L^2$-inner product over $S^d$, where $N(d, \ell)$ is the dimension of $\H_\ell^d$. It is well-known that each $Y_\ell^m$ is an eigenfunction of $-\Delta_{S^d}$ with corresponding eigenvalue $\ell(\ell + d - 1)$, \ie 
\begin{equation} \label{eq:SpHarmEigs} 
-\Delta_{S^d}Y_\ell^m(\theta) = \ell(\ell + d - 1)Y_\ell^m(\theta), \ \ 1\le m\le N(d, \ell). 
\end{equation} 
A direct computation using Green's identity, together with \eqref{eq:SpHarmEigs}, gives the following integral identity: 
\begin{equation} \label{eq:SpHarmIBP} 
\int_{S^d} \abs{\nabla_{S^d}Y_\ell^m(\theta)}^2 d\sigma_d(\theta) = \int_{S^d} \ell(\ell + d - 1) \abs{Y_\ell^m(\theta)}^2 d\sigma_d(\theta), 
\end{equation} 
where $d\sigma_d$ is the area element of $S^d$ in hyperspherical coordinates, satisfying the relation $dV = d\sigma_d(\theta)\, r^d dr$; see \eqref{eq:SCvolume}. The sum of $\H_\ell^d$ is dense in $L^2(S^d)$, thus any $\xi\in L^2(S^d)$ can be expanded in hyperspherical harmonics: 
\begin{equation} \label{eq:SpHarmExp} 
\xi(\theta) = \sum_{\ell = 0}^\infty \sum_{m = 1}^{N(d, \ell)} \hat\xi(\ell, m)Y_\ell^m(\theta), \ \ \hat\xi(\ell, m)\coloneqq \int_{S^d} \xi(\theta)\overline{Y_\ell^m(\theta)}\, d\sigma_d(\theta). 
\end{equation} 
This leads us to define the Sobolev space $H^s(S^d)$ of order $s\ge 0$ as the subspace of $L^2(S^d)$ with norm 
\begin{equation*} 
\|\xi\|_{H^s(S^d)}^2\coloneqq \sum_{\ell = 0}^\infty \sum_{m = 1}^{N(d, \ell)} \left(1 + \ell(\ell + d - 1)\right)^s\abs{\hat\xi(\ell, m)}^2. 
\end{equation*}

\begin{lemma} \label{lemma:Elliptic} 
Given $d\ge 3$ and $s\in\N$, there exists a constant $K_0 > 0$ such that for any $F\in H^{s - 1}(B)$ and $\xi\in H^{s + \frac{1}{2}}(S^d)$, the solution of 
\begin{alignat*}{2} 
\Delta w(r, \theta) & = F(r, \theta) && \qquad \tn{ in } B, \\ 
w(1, \theta) & = \xi(\theta) && \qquad \tn{ on } S^d, 
\end{alignat*} 
satisfies the following estimate: 
\begin{equation*} 
\|w\|_{H^{s + 1}(B)}\le K_0\left(\|F\|_{H^{s - 1}(B)} + \|\xi\|_{H^{s + \frac{1}{2}}(S^d)}\right). 
\end{equation*} 
\end{lemma} 
\begin{proof} 
We will prove the result for $s = 0$ only, as the proof for $s = 1, 2, \dots$ is similar. Since $\xi\in H^{\frac{1}{2}}(S^d)$, $\xi$ has a hyperspherical harmonic expansion \eqref{eq:SpHarmExp}. Let us decompose $w = \Phi + V$ where 
\begin{equation*} 
\Phi(r, \theta) = \sum_{\ell = 0}^\infty\sum_{m = 1}^{N(d, \ell)} \hat\xi(\ell, m)r^{\ell}Y_\ell^m(\theta). 
\end{equation*} 
It is clear that $\Phi$ is harmonic in $B$ and $\Phi = \xi$ on $S^d$, which means $V$ solves the following boundary value problem: 
\begin{subequations} 
\begin{alignat}{2} 
\label{eq:EllipticV1} \Delta V(r, \theta) & = F(r, \theta) && \qquad \tn{ in } B, \\ 
\label{eq:EllipticV2} V(1, \theta) & = 0 && \qquad \tn{ on } S^d. 
\end{alignat} 
\end{subequations} 
Multiplying \eqref{eq:EllipticV1} by $\overline{V}$ and integrating by parts using \eqref{eq:EllipticV2} yields 
\begin{align*} 
\|V\|_{H_0^1(B)}^2 \coloneqq \|\nabla V\|_{L^2(B)}^2 = -\langle F, V\rangle_{H^{-1}(B), H_0^1(B)},
\end{align*} 
where the latter denotes the duality pairing between $H_0^1(B)$ and $H^{-1}(B)$. Using Poincar\'e inequality and duality, there exists a constant $C_B > 0$ such that 
\begin{equation} \label{eq:Elliptic1}
\|V\|_{H^1(B)}\le C_B\|V\|_{H_0^1(B)} \le C_B\|F\|_{H^{-1}(B)}. 
\end{equation} 

It remains to show that $\|\Phi\|_{H^1(B)}$ is controlled by $\|\xi\|_{H^{\frac{1}{2}}(S^d)}$. Since $Y_\ell^m$ is orthonormal in $L^2(S^d)$, we obtain 
\begin{align*} 
\|\Phi\|_{L^2(B)}^2 & = \sum_{\ell, m} \abs{\hat\xi(\ell, m)}^2 \int_0^1\left(\int_{S^d} r^{2\ell}\abs{Y_\ell^m(\theta)}^2 d\sigma_d(\theta)\right) r^d dr \\ 
& = \sum_{\ell, m} \abs{\hat\xi(\ell, m)}^2 \int_0^1 r^{2\ell + d}\, dr = \sum_{\ell, m} \left[\frac{1}{2\ell + d + 1}\right]\abs{\hat\xi(\ell, m)}^2. 
\end{align*} 
Similarly, using \eqref{eq:SCgrad} to compute $\nabla\Phi$ and using the integral identity \eqref{eq:SpHarmIBP}, we obtain 
\begin{align*} 
\|\nabla\Phi\|_{L^2(B)}^2 & = \sum_{\ell, m} \abs{\hat\xi(\ell, m)}^2 \int_0^1\left(\int_{S^d} \left(\ell^2 r^{2(\ell - 1)}\abs{Y_\ell^m(\theta)}^2 + \frac{r^{2\ell}}{r^2}\abs{\nabla_{S^d}Y_\ell^m(\theta)}^2\right) d\sigma_d(\theta)\right) r^d dr \\ 
& \stackrel{\eqref{eq:SpHarmIBP}}{=} \sum_{\ell, m} \abs{\hat\xi(\ell, m)}^2 \int_0^1\left(\int_{S^d} r^{2(\ell - 1)}\left(\ell^2 + \ell(\ell + d - 1)\right)\abs{Y_\ell^m(\theta)}^2 d\sigma_d(\theta)\right) r^d dr \\ 
& = \sum_{\ell, m} \abs{\hat\xi(\ell, m)}^2\int_0^1 r^{2(\ell - 1) + d}\left(\ell^2 + \ell(\ell + d - 1)\right) dr \\ 
& = \sum_{\ell, m} \left[\frac{\ell^2 + \ell(\ell + d - 1)}{2\ell + d - 1}\right] \abs{\hat\xi(\ell, m)}^2. 
\end{align*} 
This allows us to estimate $\|\Phi\|_{H^1(B)}^2 = \|\Phi\|_{L^2(B)}^2 + \|\nabla\Phi\|_{L^2(B)}^2$ as 
\begin{equation} \label{eq:Elliptic2} 
\|\Phi\|_{H^1(B)}^2 = \sum_{\ell, m} \left[\frac{1}{2\ell + d + 1} + \frac{\ell^2 + \ell(\ell + d - 1)}{2\ell + d - 1}\right] \abs{\hat\xi(\ell, m)}^2 \le C_{\frac{1}{2}}^2 \|\xi\|_{H^{\frac{1}{2}}(S^d)}^2, 
\end{equation} 
for some constant $C_{\frac{1}{2}} > 0$. Finally, using $w = \Phi + V$, the desired estimate for $s = 0$ follows from \eqref{eq:Elliptic1} and \eqref{eq:Elliptic2} by choosing $K_0 = \max\left\{C_B, C_{\frac{1}{2}}\right\}$. 
\end{proof}

Next we derive recursive estimates for the right hand side of \eqref{eq:HarmExtPS1}, which is analogous to \cite[Lemma~3.2]{viator2020}. We shall use the following inequalities from \cite{nicholls2001, nicholls2004}: Given an integer $s\ge 0$ and $\delta > 0$, if $f\in C^s(S^d)$, $u\in H^s(B)$, $g\in C^{s + \frac{1}{2} + \delta}(S^d)$, and $\mu\in H^{s + \frac{1}{2}}(S^d)$, there exists a constant $M = M(d, s)$ such that 
\begin{subequations} 
\begin{align} 
\label{eq:Multiplier1} \|fu\|_{H^s(B)} & \le M(d, s)\abs{f}_{C^s}\|u\|_{H^s(B)},  \\ 
\label{eq:Multiplier2} \|g\mu\|_{H^{s + \frac{1}{2}}(S^d)} & \le M(d, s) \abs{g}_{C^{s + \frac{1}{2} + \delta}}\|\mu\|_{H^{s + \frac{1}{2}}(S^d)}. 
\end{align} 
\end{subequations} 
Here, $C^k(S^d)$ is the space of all $C^k$-functions on $S^d$ for integers $k\ge 0$ and the H\"older space for nonintegers $k$.

\begin{lemma} \label{lemma:RHS} 
Given $d\ge 3$ and $s\in\N$, let $\rho\in C^{s + 2}(S^d)$. Assume that $K_1$ and $A$ are constants so that 
\begin{equation} \label{eq:RHSAssumption} 
\|u_n\|_{H^{s + 2}(B)}\le K_1A^n \qquad \tn{ for all } n < N. 
\end{equation} 
If $A > M(d, s)\abs{\rho}_{C^{s + 2}}$, then there exists a constant $C_0 = C_0(d, s) > 0$ such that 
\begin{align*} 
\|L_1u_{N - 1}\|_{H^s(B)} & \le C_0\abs{\rho}_{C^{s + 2}}K_1A^{N - 1}, \\ 
\|L_2u_{N - 2}\|_{H^s(B)} & \le C_0\abs{\rho}_{C^{s + 2}}K_1A^{N - 1}. 
\end{align*} 
\end{lemma} 
\begin{proof} 
For notation convenience, we suppress the dependence of $M$ on $d$ and $s$. Throughout this proof we will use the fact that both $\abs{\rho}_{C^s}$ and $\abs{\rho}_{C^{s + 1}}$ are bounded above by $\abs{\rho}_{C^{s + 2}}$. First, a close inspection on $L_1u_{N - 1}$ reveals that all differential operators acting on $u_{N - 1}$ are second order, since we may write $L_1u_{N - 1}$ as follows: 
\begin{align*} 
L_1u_{N - 1} = -2\rho\cdot \Delta u_{N- 1} + (\Delta_{S^d}\rho)\cdot \frac{\del_ru_{N - 1}}{r} + 2\sum_{i = 1}^d \frac{\del_i\rho}{\eta_i}\cdot \frac{1}{h_i}\del_i\del_ru_{N - 1}. 
\end{align*} 
Estimating each term in $L_1u_{N-1}$ using the inequality \eqref{eq:Multiplier1} and noting that $\del_i\rho/\eta_i$ is a first order partial derivative on $S^d$, we obtain 
\begin{align*} 
\|L_1u_{N - 1}\|_{H^s(B)} & \le \|2\rho\Delta u_{N - 1}\|_{H^s(B)} + \|(\Delta_{S^d}\rho)\cdot r^{-1}\del_ru_{N - 1}\|_{H^s(B)} \\ 
& \qquad + 2\sum_{i = 1}^d \left\|\frac{\del_i\rho}{\eta_i}\cdot \frac{1}{h_i}\del_i\del_ru_{N - 1}\right\|_{H^s(B)} \\ 
& \stackrel{\eqref{eq:Multiplier1}}{\le} 2M\abs{\rho}_{C^s}\|\Delta u_{N - 1}\|_{H^s(B)} + M\abs{\Delta_{S^d}\rho}_{C^s}\|r^{-1}\del_ru_{N - 1}\|_{H^s(B)} \\ 
& \qquad + 2\sum_{i = 1}^d M\abs{\eta_i^{-1}\del_i\rho}_{C^s}\left\|(h_i)^{-1}\del_i\del_ru_{N - 1}\right\|_{H^s(B)} \\ 
& \le 2M\abs{\rho}_{C^s}\|u_{N - 1}\|_{H^{s + 2}(B)} + M\abs{\rho}_{C^{s + 2}}\|u_{N - 1}\|_{H^{s + 2}(B)} \\ 
& \qquad + 2\sum_{i = 1}^d M\abs{\rho}_{C^{s + 1}}\|u_{N - 1}\|_{H^{s + 2}(B)} \\ 
& \le (3 + 2d)M\abs{\rho}_{C^{s + 2}}\|u_{N - 1}\|_{H^{s + 2}(B)} \\ 
& \stackrel{\eqref{eq:RHSAssumption}}{\le} (3 + 2d)M\abs{\rho}_{C^{s + 2}}K_1A^{N - 1}. 
\end{align*} 
The estimate for $L_2u_{N - 2}$ is similar. Estimating the second term $\rho L_1u_{N - 2}$ in $L_2u_{N - 2}$ using the inequality \eqref{eq:Multiplier1} once and adapting the estimate for $L_1u_{N - 1}$ from above, we obtain 
\begin{align*} 
\|\rho L_1u_{N - 2}\|_{H^s(B)} & \le M\abs{\rho}_{C^s}\|L_1u_{N - 2}\|_{H^s(B)} \le (3 + 2d) M^2\abs{\rho}_{C^{s + 2}}^2 K_1A^{N - 2}. 
\end{align*} 
Estimating the remaining terms in $L_2u_{N - 2}$ using the inequality \eqref{eq:Multiplier1} twice and noticing that all differential operators acting on $u_{N - 2}$ are second order, we obtain 
\begin{align*} 
\|L_2u_{N - 2}\|_{H^s(B)} & \le \|\rho L_1u_{N - 2}\|_{H^s(B)} + \|\rho^2\Delta u_{N - 2}\|_{H^s(B)} \\ 
& \qquad + \sum_{i = 1}^d \left\|\left(\frac{\del_i\rho}{\eta_i}\right)^2\left[\del_r^2u_{N - 2} + 2r^{-1}\del_ru_{N - 2}\right]\right\| \\ 
& \stackrel{\eqref{eq:Multiplier1}}{\le} (3 + 2d) M^2\abs{\rho}_{C^{s + 2}}^2 K_1A^{N - 2} + M^2\abs{\rho}_{C^s}^2\|\Delta u_{N - 2}\|_{H^s(B)} \\ 
& \qquad + \sum_{i = 1}^d M^2\abs{\eta_i^{-1}\del_i\rho}_{C^s}^2\left(\|\del_r^2u_{N - 2}\|_{H^s(B)} + 2\|r^{-1}\del_ru_{N - 2}\|_{H^s(B)}\right) \\ 
& \le (3 + 2d) M^2\abs{\rho}_{C^{s + 2}}^2 K_1A^{N - 2} + M^2\abs{\rho}_{C^s}^2\|u_{N - 2}\|_{H^{s + 2}(B)} \\ 
& \qquad + \sum_{i = 1}^d M^2\abs{\rho}_{C^{s + 1}}^2\left(3\|u_{N - 2}\|_{H^{s + 2}(B)}\right) \\ 
& \le (3 + 2d) M^2\abs{\rho}_{C^{s + 2}}^2 K_1A^{N - 2} + (1 + 3d) M^2\abs{\rho}_{C^{s + 2}}^2\|u_{N - 2}\|_{H^{s + 2}(B)} \\ 
& \stackrel{\eqref{eq:RHSAssumption}}{\le} (3 + 2d) M^2\abs{\rho}_{C^{s + 2}}^2 K_1A^{N - 2} + (1 + 3d) M^2\abs{\rho}_{C^{s + 2}}^2 K_1A^{N - 2} \\ 
& \le (4 + 5d) M\abs{\rho}_{C^{s + 2}} K_1A^{N - 1}, 
\end{align*} 
assuming $A > M\abs{\rho}_{C^{s + 2}}$. Finally, the desired result follows by choosing $C_0 = (4 + 5d)M$. 
\end{proof} 

Finally, we justify the convergence of \eqref{eq:HarmExtPS} for sufficiently small $\vareps > 0$.  

\begin{theorem} \label{thm:HarmExtTran_Main} 
Given $d\ge 3$ and $s\in\N$, if $\rho\in C^{s + 2}(S^d)$ and $\xi\in H^{s + \frac{3}{2}}(S^d)$, there exists constants $C_0$ and $K_0$ and a unique solution $u_\vareps$ of \eqref{eq:HarmExtTran} satisfying \eqref{eq:HarmExtPS} such that 
\begin{equation} \label{eq:HarmExtTran_Main} 
\|u_n\|_{H^{s + 2}(B)}\le K_0\|\xi\|_{H^{s + \frac{3}{2}}(S^d)}A^n 
\end{equation} 
for any $A > \max\{2K_0C_0\abs{\rho}_{C^{s + 2}}, M\abs{\rho}_{C^{s + 2}}\}$. 
\end{theorem} 
\begin{proof} 
The proof is similar to \cite[Theorem~3.3]{viator2020}. We proceed by induction. The case $n = 0$ follows from applying Lemma \ref{lemma:Elliptic} to $u_0$ satisfying \eqref{eq:HarmExtPS0}, with $F\equiv 0$: 
\begin{align*} 
\|u_0\|_{H^{s + 2}(B)}\le K_0\|\xi\|_{H^{s + \frac{3}{2}}(S^d)} = K_0\|\xi\|_{H^{s + \frac{3}{2}}(S^d)}A^0. 
\end{align*} 
Suppose inequality \eqref{eq:HarmExtTran_Main} holds for $n < N$. Applying Lemma \ref{lemma:Elliptic} to $u_N$ satisfying \eqref{eq:HarmExtPS1}, with $F = L_1u_{N - 1} + L_2u_{N - 2}$, we obtain 
\begin{align*} 
\|u_N\|_{H^{s + 2}(B)}\le K_0\left(\|L_1u_{N - 1}\|_{H^s(B)} + \|L_2u_{N - 2}\|_{H^s(B)}\right). 
\end{align*} 
Using Lemma \ref{lemma:RHS} with $K_1 = K_0\|\xi\|_{H^{s + \frac{3}{2}}(S^d)}$ from the induction hypothesis, we may bound $\|L_1u_{N - 1}\|_{H^s(B)}$ and $\|L_2u_{N - 2}\|_{H^s(B)}$ so that 
\begin{align*} 
\|u_N\|_{H^{s + 2}(B)} \le K_0\left(2C_0\abs{\rho}_{C^{s + 2}} K_1A^{N - 1}\right) & = 2K_0^2C_0\abs{\rho}_{C^{s + 2}}\|\xi\|_{H^{s + \frac{3}{2}}(S^d)}A^{N - 1} \\ 
& \le K_0\|\xi\|_{H^{s + \frac{3}{2}}(S^d)} A^N, 
\end{align*} 
thereby establishing \eqref{eq:HarmExtTran_Main} for $n = N$, provided $A > 2K_0C_0\abs{\rho}_{C^{s + 2}}$. 
\end{proof}


\section{Analyticity of DNO and Steklov Eigenvalues} \label{sec:DNO} 
This last section is devoted to proving analyticity of DNO and subsequently the analyticity of Steklov eigenvalues, with respect to the perturbation parameter $\vareps$. Recall that the DNO $G_{\rho, \vareps}\colon H^{s + \frac{1}{2}}(\del\Omega_\vareps)\to H^{s - \frac{1}{2}}(\del\Omega_\vareps)$ is given by 
\begin{equation*}
G_{\rho, \vareps}\xi = \n_{\rho, \vareps}\cdot \nabla v\rvert_{r = 1 + \vareps\rho}, 
\end{equation*} 
where $\n_{\rho, \vareps}$ is the unit normal vector of $\del\Omega_\vareps$ and $v$ is the harmonic extension of $\xi$ from $\del\Omega_\vareps$ to $\Omega_\vareps$, satisfying \eqref{eq:HarmExt}. We may represent the hypersurface $\del\Omega_\vareps$ as the zero level set of the implicit function $J$, \ie 
\begin{align*} 
\del\Omega_\vareps = \left\{(r, \hat\theta)\colon J(r, \hat\theta) = \frac{r}{1 + \vareps\rho(\hat\theta)} - 1 = 0\right\}. 
\end{align*} 
A normal vector to $\del\Omega_\vareps$ is given by the gradient $\nabla J$. Using \eqref{eq:SCgrad}, we obtain  
\begin{align*}
\nabla J & = \frac{1}{1 + \vareps\rho}\vec{r} + \frac{1}{r}\sum_{i = 1}^d \frac{1}{\eta_i}\cdot \frac{-\vareps r\del_i\rho}{(1 + \vareps\rho)^2}\, \vec{\theta}_i = \frac{1}{(1 + \vareps\rho)^2}\left[(1 + \vareps\rho)\vec{r} - \vareps\sum_{i = 1}^d \frac{\del_i\rho}{\eta_i}\vec{\theta}_i\right]. 
\end{align*} 
It follows that 
\begin{align*} 
\n_{\rho, \vareps} = \frac{\nabla J}{\|\nabla J\|_2} & = \left((1 + \vareps\rho)^2 + \vareps^2\sum_{i = 1}^d \frac{(\del_i\rho)^2}{\eta_i^2}\right)^{-1/2}\left[(1 + \vareps\rho)\, \vec{r} - \vareps\sum_{i = 1}^d \frac{\del_i\rho}{\eta_i}\vec{\theta}_i\right] \\ 
& = M_{\rho, \vareps}\left[(1 + \vareps\rho)\, \vec{r} - \vareps\sum_{i = 1}^d \frac{\del_i\rho}{\eta_i}\vec{\theta}_i\right] 
\end{align*} 
and the DNO is given by 
\begin{align*} 
G_{\rho, \vareps}\xi & = M_{\rho, \vareps}\left[(1 + \vareps\rho)\del_rv - \frac{\vareps}{r}\sum_{i = 1}^d \frac{\del_i\rho}{\eta_i^2}\del_iv\right]\bigg\rvert_{r = 1 + \vareps\rho}
\end{align*} 

Next we transform the DNO from $\del\Omega_\vareps$ to $S^d$ using the change of variables \eqref{eq:CoV}. Setting 
\begin{equation*} 
u_\vareps(r', \theta') = v\left((1 + \vareps\rho(\theta'))r', \theta'\right) 
\end{equation*} 
and dropping the primes on the transformed variables, we obtain 
\begin{align*} 
G_{\rho, \vareps}\xi & = M_{\rho, \vareps}\left[\del_ru_\vareps - \frac{\vareps}{r(1 + \vareps\rho)}\sum_{i = 1}^d \frac{\del_i\rho}{\eta_i^2}\left(\del_iu_\vareps - \frac{\vareps r\del_i\rho}{(1 + \vareps\rho)}\del_ru_\vareps\right)\right]\bigg\rvert_{r = 1} \\ 
& = M_{\rho, \vareps}\widehat G_{\rho, \vareps}\xi, 
\end{align*} 
where 
\begin{equation} \label{eq:DNOCoV} 
\begin{aligned} 
\widehat G_{\rho, \vareps}\xi & = \del_ru_\vareps(1, \theta) - \frac{\vareps}{(1 + \vareps\rho)}\sum_{i = 1}^d \frac{\del_i\rho}{\eta_i^2}\del_iu_\vareps(1, \theta) + \frac{\vareps^2}{(1 + \vareps\rho)^2}\sum_{i = 1}^d \frac{(\del_i\rho)^2}{\eta_i^2}\del_ru_\vareps(1, \theta). 
\end{aligned} 
\end{equation} 
Since $M_{\rho, \vareps}$ is clearly analytic near $\vareps = 0$, we need only show the analyticity of $\widehat G_{\rho, \vareps}$ near $\vareps = 0$. Multiplying \eqref{eq:DNOCoV} by $(1 + \vareps\rho)^2$ and collecting terms of order $\vareps$ and $\vareps^2$, we see that $\widehat G_{\rho, \vareps}\xi$ satisfies 
\begin{equation} \label{eq:DNOCoV_Rearrange} 
\widehat G_{\rho, \vareps}\xi = \del_ru_\vareps(1, \theta) + \vareps\Big[(T_1u_\vareps)(1, \theta) - 2\rho\widehat G_{\rho, \vareps}\xi\Big] + \vareps^2\Big[(T_2u_\vareps)(1, \theta) - \rho^2\widehat G_{\rho, \vareps}\xi\Big], 
\end{equation} 
where 
\begin{align*} 
T_1u_\vareps & = 2\rho \del_ru_\vareps - \sum_{i = 1}^d \frac{\del_i\rho}{\eta_i^2}\del_iu_\vareps, \\ 
T_2u_\vareps & = \rho^2\del_ru_\vareps - \sum_{i = 1}^d \frac{\rho\del_i\rho}{\eta_i^2}\del_iu_\vareps + \sum_{i = 1}^d \frac{(\del_i\rho)^2}{\eta_i^2}\del_ru_\vareps. 
\end{align*} 

We now formally expand $\widehat G_{\rho, \vareps}$ as a power series in $\vareps$: 
\begin{equation} \label{eq:DNOCoV_PS} 
\widehat G_{\rho, \vareps}\xi = \sum_{n = 0}^\infty \vareps^n\widehat G_{\rho, n}\xi. 
\end{equation} 
Substituting \eqref{eq:DNOCoV_PS} and \eqref{eq:HarmExtPS} into \eqref{eq:DNOCoV_Rearrange} and collecting terms in powers of $\vareps$, we see that $\widehat G_{\rho, 0}$ satisfies 
\begin{equation} \label{eq:DNOCoV_PS0} 
\widehat G_{\rho, 0}\xi = \del_ru_0(1, \theta) 
\end{equation} 
and $\widehat G_{\rho, n}$ for $n = 1, 2, \dots$ satisfies the following recursive formula: 
\begin{equation} \label{eq:DNOCoV_PS1} 
\widehat G_{\rho, n}\xi = \del_ru_n(1, \theta) + (T_1u_{n - 1})(1, \theta) - 2\rho\widehat G_{\rho, n - 1}\xi + (T_2u_{n - 2})(1, \theta) - \rho^2\widehat G_{\rho, n - 2}\xi. 
\end{equation}

The following theorem justifies the convergence of \eqref{eq:DNOCoV_PS} for suitably small $\vareps > 0$ and thereby proves Theorem \ref{thm:DNOAnalytic} for $d\ge 3$. 

\begin{theorem} \label{thm:DNOCoV_Analytic} 
Given $d\ge 3$ and $s\in\N$, if $\rho\in C^{s + 2}(S^d)$ and $\xi\in H^{s + \frac{3}{2}}(S^d)$, then there exists constants $K_1 = K_1(B, d, s) > 0$ and $\alpha = \alpha(d) > 1$ such that 
\begin{equation} \label{eq:DNOCoV_Analytic}
\|\widehat G_{\rho, n}\xi\|_{H^{s + \frac{1}{2}}(S^d)}\le K_1\|\xi\|_{H^{s + \frac{3}{2}}(S^d)}A^n 
\end{equation} 
for $A > \alpha\max\{2K_0C_0\abs{\rho}_{C^{s + 2}}, M\abs{\rho}_{C^{s + 2}}\}$. 
\end{theorem} 
\begin{proof} 
Throughout this proof we will use two facts: (1) Both $\abs{\rho}_{C^s}$ and $\abs{\rho}_{C^{s + 1}}$ are bounded above by $\abs{\rho}_{C^{s + 2}}$, and (2) if $P$ be any first order differential operator on $S^d$, then the continuity of the trace operator on $H^{s + 1}(B)$ implies the existence of a constant $C_B > 0$ such that 
\begin{equation} \label{eq:Trace} 
\begin{aligned} 
\|Pu_n\|_{H^{s + \frac{1}{2}}(S^d)} \le C_B\|Pu_n\|_{H^{s + 1}(B)} \le C_B\|u_n\|_{H^{s + 2}(B)} \le C_BK_0\|\xi\|_{H^{s + \frac{3}{2}}(S^d)} A^n, 
\end{aligned} 
\end{equation} 
where the last inequality follows from Theorem \ref{thm:HarmExtTran_Main}. 

We proceed by induction. Set $K_1 = \beta C_BK_0$ for some constant $\beta > 1$. The case $n = 0$ follows directly from \eqref{eq:DNOCoV_PS0} and \eqref{eq:Trace}: 
\begin{align*} 
\|\widehat G_{\rho, 0}\xi\|_{H^{s + \frac{1}{2}}(S^d)} & \le C_BK_0\|\xi\|_{H^{s + {\frac{3}{2}}}(S^d)}A^0 < K_1\|\xi\|_{H^{s + {\frac{3}{2}}}(S^d)}A^0. 
\end{align*} 
Now suppose that \eqref{eq:DNOCoV_Analytic} holds for $n < N$. From the recursive formula \eqref{eq:DNOCoV_PS1}, we estimate each term in $\widehat G_{\rho, N}\xi$ individually. The first term can be estimated using \eqref{eq:Trace} with $n = N$: 
\begin{align*} 
\|\del_ru_N\|_{H^{s + \frac{1}{2}}(S^d)} & \le C_BK_0\|\xi\|_{H^{s + \frac{3}{2}}(S^d)}A^N. 
\end{align*} 
To estimate the terms $2\rho\widehat G_{\rho, N - 1}\xi$ and $\rho^2\widehat G_{\rho, N - 2}\xi$, we use the inequality \eqref{eq:Multiplier2} with $\delta = 1/2$ together with the induction hypothesis:  
\begin{align*} 
2\|\rho\widehat G_{\rho, N - 1}\xi\|_{H^{s + \frac{1}{2}}(S^d)} \stackrel{\eqref{eq:Multiplier2}}{\le} 2M\abs{\rho}_{C^{s + 1}}\|\widehat G_{\rho, N - 1}\xi\|_{H^{s + \frac{1}{2}}(S^d)} \le 2M \abs{\rho}_{C^{s + 2}}K_1\|\xi\|_{H^{s + \frac{3}{2}}(S^d)}A^{N - 1} \\ 
\|\rho^2\widehat G_{\rho, N - 2}\xi\|_{H^{s + \frac{1}{2}}(S^d)} \stackrel{\eqref{eq:Multiplier2}}{\le} M^2\abs{\rho}_{C^{s + 1}}^2\|\widehat G_{\rho, N - 2}\xi\|_{H^{s + \frac{1}{2}}(S^d)} \le M^2\abs{\rho}_{C^{s + 2}}^2K_1\|\xi\|_{H^{s + \frac{3}{2}}(S^d)}A^{N - 2}. 
\end{align*} 
To estimate the terms $T_1u_{N - 1}(1, \theta)$ and $T_2u_{N - 1}(1, \theta)$, we observe that the differential operators acting on $u_{N - 1}$ and $u_{N-2}$ are first order on $S^d$. Using the inequality \eqref{eq:Multiplier2} with $\delta = 1/2$ and \eqref{eq:Trace}, we obtain 
\begin{align*} 
\|T_1u_{N - 1}\|_{H^{s + \frac{1}{2}}(S^d)} & \le \|2\rho\del_ru_{N - 1}\|_{H^{s + \frac{1}{2}}(S^d)} + \sum_{i = 1}^d \left\|\frac{\del_i\rho}{\eta_i}\cdot \frac{\del_iu_{N - 1}}{\eta_i}\right\|_{H^{s + \frac{1}{2}}(S^d)} \\ 
& \stackrel{\eqref{eq:Multiplier2}}{\le} 2M\abs{\rho}_{C^{s + 1}}\|\del_ru_{N - 1}\|_{H^{s + \frac{1}{2}}(S^d)} + \sum_{i = 1}^d M\abs{\rho}_{C^{s + 2}}\|\eta_i^{-1}\del_iu_{N - 1}\|_{H^{s + \frac{1}{2}}(S^d)} \\ 
& \stackrel{\eqref{eq:Trace}}{\le} (2 + d)M\abs{\rho}_{C^{s + 2}}C_BK_0\|\xi\|_{H^{s + \frac{3}{2}}(S^d)}A^{N - 1} \\ 
\|T_2u_{N - 2}\|_{H^{s + \frac{1}{2}}(S^d)} & \le \|\rho^2\del_ru_{N - 2}\|_{H^{s + \frac{1}{2}}(S^d)} + \sum_{i = 1}^d \left\|\rho\cdot \frac{\del_i\rho}{\eta_i}\cdot \frac{\del_iu_{N - 2}}{\eta_i}\right\|_{H^{s + \frac{1}{2}}(S^d)} \\ 
& \qquad + \sum_{i = 1}^d \left\|\left(\frac{\del_i\rho}{\eta_i}\right)^2\del_ru_{N - 2}\right\|_{H^{s + \frac{1}{2}}(S^d)} \\ 
& \stackrel{\eqref{eq:Multiplier2}}{\le} M^2\abs{\rho}_{C^{s + 1}}^2\|\del_ru_{N -2}\|_{H^{s + \frac{1}{2}}(S^d)} + \sum_{i = 1}^d M^2\abs{\rho}_{C^{s + 1}}\abs{\rho}_{C^{s + 2}}\|\eta_i^{-1}\del_iu_{N - 2}\|_{H^{s + \frac{1}{2}}(S^d)} \\ 
& \qquad + \sum_{i = 1}^d M^2\abs{\rho}_{C^{s + 2}}^2\|\del_ru_{N - 2}\|_{H^{s + \frac{1}{2}}(S^d)} \\ 
& \stackrel{\eqref{eq:Trace}}{\le} (1 + 2d)M^2\abs{\rho}_{C^{s + 2}}^2C_BK_0\|\xi\|_{H^{s + \frac{3}{2}}(S^d)}A^{N - 2}. 
\end{align*} 
Adding all the estimates for the five terms in $\widehat G_{\rho, N}\xi$ and using the assumption $A > \alpha M\abs{\rho}_{C^{s + 2}}$ for some constant $\alpha > 1$, we obtain 
\begin{align*} 
\|\widehat G_{\rho, N}\xi\|_{H^{s + \frac{1}{2}}(S^d)} & \le C_BK_0\|\xi\|_{H^{s + \frac{3}{2}}(S^d)}A^N \\ 
& \qquad + (2 + d)M\abs{\rho}_{C^{s + 2}}C_BK_0\|\xi\|_{H^{s + \frac{3}{2}}(S^d)}A^{N - 1} \\ 
& \qquad + 2M \abs{\rho}_{C^{s + 2}}K_1\|\xi\|_{H^{s + \frac{3}{2}}(S^d)}A^{N - 1} \\ 
& \qquad + (1 + 2d)M^2\abs{\rho}_{C^{s + 2}}^2C_BK_0\|\xi\|_{H^{s + \frac{3}{2}}(S^d)}A^{N - 2} \\ 
& \qquad + M^2\abs{\rho}_{C^{s + 2}}^2K_1\|\xi\|_{H^{s + \frac{3}{2}}(S^d)}A^{N - 2} \\ 
& \le C_BK_0\left[1 + \frac{2 + d}{\alpha} + \frac{1 + 2d}{\alpha^2}\right]\|\xi\|_{H^{s + \frac{3}{2}}(S^d)}A^N + K_1\left[\frac{2}{\alpha} + \frac{1}{\alpha^2}\right] \|\xi\|_{H^{s + \frac{3}{2}}(S^d)}A^N \\ 
& = C(\alpha, \beta) K_1\|\xi\|_{H^{s + \frac{3}{2}}(S^d)}A^N, 
\end{align*} 
where we use $K_1 = \beta C_BK_0$ for $\beta > 1$ for the equality and 
\begin{align*} 
C(\alpha, \beta) & = \frac{1}{\beta}\left[1 + \frac{2 + d}{\alpha} + \frac{1 + 2d}{\alpha^2}\right] + \frac{2}{\alpha} + \frac{1}{\alpha^2}. 
\end{align*} 
Choosing $\alpha > 1 + \sqrt{2} > 1$ and $\beta > \frac{\alpha^2 + \alpha(2 + d) + 2d + 1}{\alpha^2 - 2\alpha - 1} > 1$, we obtain $C(\alpha, \beta)\le 1$ and so 
\begin{align*} 
\|\widehat G_{\rho, N}\xi\|_{H^{s + \frac{1}{2}}(S^d)}\le K_1\|\xi\|_{H^{s + \frac{3}{2}}(S^d)}A^N, 
\end{align*} 
thereby completing the inductive argument for $n = N$. 
\end{proof}

\begin{proof}[Proof of Theorem \ref{thm:DNOAnalytic} and Corollary \ref{thm:Steklov}.]
The proof is identical to \cite[Corollary~1.2]{viator2020}, but we include it here for completeness. Theorem \ref{thm:DNOCoV_Analytic} shows that the expansion of the non-normalized DNO $\widehat G_{\rho, \varepsilon}$ given in \eqref{eq:DNOCoV_PS} is uniformly convergent for small $\varepsilon$. It follows that $G_{\rho, \varepsilon} \colon H^{s+\frac 3 2}(S^d) \to H^{s + \frac 1 2}(S^d)$ is analytic for small $\varepsilon$.  The DNO operator $ G_{\rho, \varepsilon}\colon L^2(S^d) \to L^2(S^d)$ is self-adjoint \cite{arendt2014}, hence closed. 
The result now follows from \cite[Ch. 7, Thm 1.8, p. 370]{Kato}. 
\end{proof}

\subsubsection*{Acknowledgements.} The authors would like to thank John Gemmer, Jason Parsley, and Nsoki Mavinga for helpful insights.

\printbibliography 

\end{document}